\documentclass{amsart}
\usepackage{amssymb}
\usepackage{amsfonts}
\usepackage{amssymb}
\usepackage{amsmath}
\usepackage{amsthm}
\usepackage{enumerate}
\usepackage{tabularx}
\usepackage{centernot}
\usepackage{mathtools}
\usepackage{stmaryrd}
\usepackage{amsthm,amssymb}
\usepackage{etoolbox}
\usepackage{tikz}
\usepackage{amssymb}
\usetikzlibrary{matrix}
\usepackage{tikz-cd}
\usepackage{tikz}
\usepackage{marginnote}
\definecolor{mygray}{gray}{0.85}
\usepackage[linecolor=black, backgroundcolor=mygray,colorinlistoftodos,prependcaption,textsize=small]{todonotes}
\usepackage{xargs}

\renewcommand{\leq}{\leqslant}
\renewcommand{\geq}{\geqslant}

\renewcommand{\trianglelefteq}{\trianglelefteqslant}

\makeatletter
\def\subsection{\@startsection{subsection}{3}%
  \z@{.5\linespacing\@plus.7\linespacing}{.3\linespacing}%
  {\bfseries\centering}}
\makeatother

\makeatletter
\def\subsubsection{\@startsection{subsubsection}{3}%
  \z@{.5\linespacing\@plus.7\linespacing}{.3\linespacing}%
  {\centering}}
\makeatother

\makeatletter
\def\myfnt{\ifx\protect\@typeset@protect\expandafter\footnote\else\expandafter\@gobble\fi}
\makeatother

\newtheorem{theorem}{Theorem}

\newtheorem{corollary}[theorem]{Corollary}
\newtheorem{definition}[theorem]{Definition}
\newtheorem{lemma}[theorem]{Lemma}
\newtheorem{proposition}[theorem]{Proposition}

\newtheorem{convention}[theorem]{Convention}

\newtheorem{fact}[theorem]{Fact}
\newtheorem{definition_proposition}[theorem]{Definition/Proposition}

\newtheorem{notation}[theorem]{Notation}

\newcounter{claimcounter}
\numberwithin{claimcounter}{theorem}

\begin{document}

\begin{abstract} 
We complement the characterization of the graph products of cyclic groups $G(\Gamma, \mathfrak{p})$ admitting a Polish group topology of \cite{shelah&paolini} with the following result. Let $G = G(\Gamma, \mathfrak{p})$, then the following are equivalent:
\begin{enumerate}[(i)]
	\item there is a metric on $\Gamma$ which induces a separable topology in which $E_{\Gamma}$ is closed;
	\item $G(\Gamma, \mathfrak{p})$ is embeddable into a Polish group;
	\item $G(\Gamma, \mathfrak{p})$ is embeddable into a non-Archimedean Polish group.
	\end{enumerate}
We also construct left-invariant separable group ultrametrics for $G = G(\Gamma, \mathfrak{p})$ and $\Gamma$ a closed graph on the Baire space, which is of independent interest.
\end{abstract}

\title{Group Metrics for Graph Products of Cyclic Groups}
\thanks{Partially supported by European Research Council grant 338821. No. F1668 on Shelah's publication list.}

\author{Gianluca Paolini}
\address{Einstein Institute of Mathematics,  The Hebrew University of Jerusalem, Israel}

\author{Saharon Shelah}
\address{Einstein Institute of Mathematics,  The Hebrew University of Jerusalem, Israel \and Department of Mathematics,  Rutgers University, U.S.A.}

\maketitle


\section{Introduction}

	\begin{definition}\label{def_cyclic_prod} Let $\Gamma = (V, E)$ be a graph and $\mathfrak{p}: V \rightarrow \{ p^n : p \text{ prime, } n \geq 1 \} \cup \{ \infty \}$ a graph coloring. We define a group $G(\Gamma, \mathfrak{p})$ with the following presentation:
	$$ \langle V \mid a^{\mathfrak{p}(a)} = 1, \; bc = cb : \mathfrak{p}(a) \neq \infty \text{ and }  b E c \rangle.$$
\end{definition}
We call the group $G(\Gamma, \mathfrak{p})$ the {\em $\Gamma$-product\footnote{Notice that this is consistent with the general definition of graph products of groups from \cite{green}. In fact every graph product of cyclic groups can be represented as $G(\Gamma, \mathfrak{p})$ for some $\Gamma$ and $\mathfrak{p}$ as above.} of the cyclic groups} $\{ C_{\mathfrak{p}(v)} : v \in \Gamma \}$, or simply the {\em graph product of} $(\Gamma, \mathfrak{p})$. These groups have received much attention in combinatorial and geometric group theory. In \cite{shelah&paolini} the authors characterized the graph products of cyclic groups admitting a Polish group topology, showing that $G$ has to have the form $G_1 \oplus G_2$ with $G_1$ a countable graph product of cyclic groups and $G_2$ a direct sum of finitely many continuum sized vector spaces over a finite field. In the present study we complement the work of \cite{shelah&paolini} with the following results: 

	\begin{theorem}\label{theorem1} Let $\Gamma = (\omega^{\omega}, E)$ be a graph and $\mathfrak{p}: V \rightarrow \{ p^n : p \text{ prime, } n \geq 1 \} \cup \{ \infty \}$ a graph coloring. Suppose further that $E$ is closed in the Baire space $\omega^{\omega}$, and that $\mathfrak{p}(\eta)$ depends only on $\eta(0)$. Then $G = G(\Gamma, \mathfrak{p})$ admits a left-invariant separable group ultrametric extending the standard metric on the Baire space.
\end{theorem}

%

	\begin{theorem}\label{embed_for_cyclic} Let $G = G(\Gamma, \mathfrak{p})$, then the following are equivalent:
\begin{enumerate}[(a)]
	\item there is a metric on $\Gamma$ which induces a separable topology in which $E_{\Gamma}$ is closed;
	\item $G$ is embeddable into a Polish group;
	\item $G$ is embeddable into a non-Archimedean Polish group;
\end{enumerate}
\end{theorem}

	\begin{corollary}\label{random_graph} Let $G = G(\Gamma, \mathfrak{p})$, then the following are equivalent:
\begin{enumerate}[(a)]
	\item there is a metric on $\Gamma$ which induces a separable topology in which $E_{\Gamma}$ is closed;
	\item $G$ is embeddable into the automorphism group of the random graph;
	\item $G$ is embeddable into the automorphism group of Hall's universal locally finite group.
\end{enumerate}
\end{corollary}

	The condition(s) occurring in Theorem 3 and Corollary \ref{random_graph} fail e.g. for the $\aleph_1$-half graph $\Gamma = \Gamma(\aleph_1)$, i.e. the graph on vertex set $\{ a_{\alpha} : \alpha < \aleph_1 \} \cup  \{ b_{\beta}: \beta < \aleph_1 \}$ with edge relation defined as $a_\alpha E_{\Gamma} b_\beta$ if and only if  $\alpha < \beta$.

	Theorem \ref{theorem1} is of independent interest and generalizes results on left-invariant group metrics on free groups on continuum many generators, see \cite{ding_gao}, \cite{ding_gao_1} and \cite{gao_1}.
\section{Proofs of the Theorems}

	\begin{convention} In Definition \ref{def_cyclic_prod} it is usually assumed that for every $a \in \Gamma$ we have $\{ a, a \} \not\in E_{\Gamma}$. In order to make our proofs more transparent we will diverge from this convention and assume that our graphs $\Gamma$ are such that $a \in \Gamma$ implies $aE_{\Gamma}a$. This is of course irrelevant from the point of view of the group $G = G(\Gamma, \mathfrak{p})$, since an element $a \in G$ always commutes with itself. 
\end{convention}

	\begin{proposition}\label{propo_closed} Let $G$ be a separable topological group which is metrizable (resp. ultrametrizable) by the metric $d$ and $V \subseteq G$. Then the metric (resp. ultrametric) $d \restriction V \times V$ makes $V$ into a separable space such that for every group term $\sigma$ the set $\{ \bar{a} \in V^{|\sigma|} : G \models \sigma(\bar{a}) = e \}$ is closed in the induced topology.
\end{proposition}

	\begin{proof} For every group term $\sigma$ the map $\bar{a} \mapsto \sigma(\bar{a})$ is continuous. Thus the set $\{ \bar{a} \in G^{|\sigma|} : G \models \sigma(\bar{a}) = e \}$ is closed in $(G, d)$, and so the set:
	$$\{ \bar{a} \in V^{|\sigma|} : G \models \sigma(\bar{a}) = e \} = \{ \bar{a} \in G^{|\sigma|} : G \models \sigma(\bar{a}) = e \} \cap V^{|\sigma|}$$
is closed in $(V, d \restriction V \times V)$. 
\end{proof}

	\begin{notation}
	\begin{enumerate}[(1)]
	\item Given a graph $\Gamma = (V, E)$ and a set $R$, by a map $h : \Gamma \rightarrow R$ we mean a map with domain $V$. Furthermore, given a map $h : \Gamma \rightarrow R$ we let $h(E) = \{ \{ h(a), h(b) \} : \{ a, b \} \in E \}$.
	\item Given $\eta \in X^\omega$, $n < \omega$ and $\nu \in X^{n}$, we write $\nu \triangleleft \eta$ to mean that $\eta \restriction n = \nu$. 
	\item Given $\eta \neq \eta' \in X^\omega$, we let $\eta \wedge \eta'$ be the unique $\nu \in X^{n}$ such that $\nu \triangleleft \eta$, $\nu \triangleleft \eta'$ and $n$ is maximal, and in this case we also let $lg(\eta \wedge \eta') = lg(\nu) = n$.
	\item Given a topological space $X$ and $Y \subseteq X$, we denote by $\overline{Y}$ the topological closure of $Y$ in $X$. Also, we denote by $\Delta_{X}$ the set $\{ (x, x) : x \in X \}$.
\end{enumerate}
\end{notation}

	\begin{lemma}\label{metric_implies_ultrametric} Let $\Gamma$ be a graph and $\mathfrak{p}: \Gamma \rightarrow \omega$ a graph colouring. Suppose that $\Gamma$ admits a separable metric $d$ which makes $E_{\Gamma}$ closed in the induced topology. Then:
	\begin{enumerate}[(1)]
	\item $\Gamma$ admits an ultrametric $d'$ with the same properties;
	\item there exists a one-to-one map $h: \Gamma \rightarrow \omega^\omega$ and a map $\mathfrak{p}^*: \omega^\omega \rightarrow \omega$ such that:
	\begin{enumerate}[(a)]
	\item $\overline{h(E_{\Gamma}) \cup \Delta_{\omega^{\omega}}} \cap h(\Gamma \times \Gamma) = h(E_{\Gamma})$;
	\item $\mathfrak{p}(a) = \mathfrak{p}^*(h(a))$, for every $a \in \Gamma$;
	\item $\eta_1(0) = \eta_2(0)$ if and only if $\mathfrak{p}^*(\eta_1) = \mathfrak{p}^*(\eta_2)$, for every $\eta_1, \eta_2 \in \omega^\omega$.
\end{enumerate}
\end{enumerate}
\end{lemma}

	\begin{proof} Let $(\Gamma, \mathfrak{p})$ and $d$ be as in the statement of the lemma. If $\Gamma$ is countable the lemma is clearly true. Assume then that $\Gamma$ is uncountable. Let $D \subseteq \Gamma$ be a countable dense set of $(\Gamma, d)$, and $\leq_D$ a well-order of $D$ of order type $\omega$. Renaming the elements of $\Gamma$ we can assume that $D = \omega$ and $\leq_D$ is the usual order of the natural numbers. For $a \in \Gamma$ we define $\eta_a \in \omega^{\omega}$ by letting:
	$$\eta_a(n) = \begin{cases} \mathfrak{p}(a) \;\;\;\;\;\;\;\; \text{ if } n = 0  \\
	x(a, n) \;\;\;\; \text{ if } n > 0,
\end{cases}$$
where:
	\begin{enumerate}[(i)]
	\item $x(a, n)$ is at distance $< 1/4^n$ from $a$;
	\item $x(a, n)$ is minimal under the condition (i).
\end{enumerate}
	We define $d': \Gamma \times \Gamma \rightarrow \mathbb{R}_{\geq 0}$ such that:
	$$d'(a, b) = \frac{1}{lg(\eta_a \wedge \eta_b) + 2}.$$
Clearly $d'$ is an ultrametric. We verify $d'$ is as required.
\begin{enumerate}[$(*)_1$]
	\item $(\Gamma, d')$ is separable. 
\end{enumerate}
For each $\nu \in \omega^{< \omega}$ choose $a_{\nu}$ such that $\nu \triangleleft \eta_{a_{\nu}}$, if possible, and arbitrarily otherwise. Let $D' = \{ a_{\nu} : \nu \in \omega^{< \omega} \}$. We claim that $D'$ is dense in $(\Gamma, d')$. This suffices, since obviously $D'$ is a countable subset of $\Gamma$. Let then $b \in \Gamma$ and $\varepsilon > 0$, we shall find $a \in D'$ such that $d'(a, b) < \varepsilon$. Choose $n > 0$ such that $1/(n+2) < \varepsilon$, and let $\nu = \eta_b \restriction n$. Now, by the choice of $\nu$, $a_{\nu} \in D'$ and $\nu \triangleleft \eta_{a_{\nu}}$. Furthermore, clearly $\nu \trianglelefteq \eta_{a_{\nu}} \wedge \eta_b$, and so $lg(\eta_{a_{\nu}} \wedge \eta_b) \geq lg(\nu) = n$. Thus we have:
$$d'(a_{\nu}, b) = \frac{1}{lg(\eta_{a_{\nu}} \wedge \eta_b) + 2} \leq \frac{1}{n+2} < \varepsilon.$$
\begin{enumerate}[$(*)_2$]
	\item $E_{\Gamma}$ is closed in $(\Gamma, d')$. 
\end{enumerate}
Let $a, b \in \Gamma$ and suppose that $\{ a, b \} \not\in E_{\Gamma}$. Since $E_{\Gamma}$ is closed in $(\Gamma, d)$, there is $\varepsilon \in (0, 1)$ such that:
\begin{equation}
a', b' \in \Gamma, \; d(a, a') < \varepsilon, \; d(b, b') < \varepsilon \;\; \Rightarrow \;\; \{ a', b' \} \not\in E_{\Gamma}.
\end{equation}
Let $n < \omega$ be such that $n > 1$ and $1/n < \varepsilon$, we shall prove that:
\begin{equation}
a', b' \in \Gamma, \; d'(a, a') < \frac{1}{n+2}, \; d'(b, b') < \frac{1}{n+2} \;\; \Rightarrow \;\; \{ a', b' \} \not\in E_{\Gamma}.
\end{equation}
Now, for any $a'$ as in (2) we have that $lg(\eta_a \wedge \eta_{a'}) > n$, and so $\eta_{a}(n) = \eta_{a'}(n)$. Hence:
$$d(a, a') \leq d(a, \eta_{a}(n)) + d(a', \eta_{a}(n)) < \frac{1}{4^n} + \frac{1}{4^n} <  1/n < \varepsilon.$$
Using the same argument we see that for any $b'$ as in (2) we have that $d(b, b') < \varepsilon$, and so by (1) we conclude that $\{ a', b' \} \not\in E_{\Gamma}$, as wanted.
\begin{enumerate}[$(*)_3$]
	\item The map $h: \Gamma \rightarrow \omega^\omega$ such that $h(a) = \eta_a$ is one-to-one.
\end{enumerate} 
If $\eta_a = \eta_b$, then:
$$lim_{n \rightarrow \infty} \eta_{a}(n) = a = lim_{n \rightarrow \infty} \eta_{b}(n) = b.$$
\begin{enumerate}[$(*)_4$]
	\item $\overline{h(E_{\Gamma})} \cap h(\Gamma \times \Gamma) = h(E_{\Gamma})$.
\end{enumerate} 
Notice that for $(c_n)_{n < \omega} \in \Gamma^{\omega}$ and $c \in \Gamma$ we have:
	$$lim_{n \rightarrow \infty} \eta_{c_n} = \eta_c \; \Rightarrow \; lim_{n \rightarrow \infty} c_n = c \;\; \text{in }(\Gamma, d').$$
	Thus, if we have:
	$$lim_{n \rightarrow \infty} \eta_{a_n} = \eta_a, \; lim_{n \rightarrow \infty} \eta_{b_n} = \eta_b \text{ and } \bigwedge_{n < \omega} a_n E_{\Gamma} b_n,$$
then $aE_{\Gamma}b$, since $E_{\Gamma}$ is closed in $(\Gamma, d')$.
\begin{enumerate}[$(*)_5$]
	\item Let $\mathfrak{p}^*: \omega^\omega \rightarrow \omega$ be such that:
	$$\mathfrak{p}^*(\eta) = \begin{cases} \eta(0) \;\;\;\; \text{ if } \exists \eta_a (\eta_a(0) = \eta(0)) \\
	1 \;\;\;\;\;\;\;\;\; \text{ otherwise}.
\end{cases}$$
\end{enumerate} 
Then the map $\mathfrak{p}^*$ is clearly as wanted.
\end{proof}

	We need some basic word combinatorics for $G(\Gamma, \mathfrak{p})$.
	
	\begin{definition} Let $(\Gamma, \mathfrak{p})$ be as usual and $G = G(\Gamma, \mathfrak{p})$.
	\begin{enumerate}[(1)]
	\item A word $w$ in the alphabet $\Gamma$ is a sequence $(a_1^{\alpha_1}, ..., a_k^{\alpha_k})$, with $a_i \neq a_{i+1} \in \Gamma$, for $i = 1, ..., k-1$, and $\alpha_1, ..., \alpha_k \in \mathbb{Z} - \{0 \}$.
	\item We denote words simply as $a_1^{\alpha_1} \cdots a_k^{\alpha_k}$ instead of $(a_1^{\alpha_1}, ..., a_k^{\alpha_k})$.
	\item We call each $a_i^{\alpha_i}$ a syllable of the word $a_1^{\alpha_1} \cdots a_k^{\alpha_k}$.
	\item We say that the word $a_1^{\alpha_1} \cdots a_k^{\alpha_k}$ spells the element $g \in G$ if $ G \models g = a_1^{\alpha_1} \cdots a_k^{\alpha_k}$.
	\item We say that the word $w$ is reduced if there is no word with fewer syllables which spells the same element of $G$.
	\item We say that the consecutive syllables $a_i^{\alpha_i}$ and $a_{i+1}^{\alpha_{i+1}}$ are adjacent if $a_iE_{\Gamma}a_{i+1}$.
	\item We say that the word $w$ is a normal form for $g$ if it spells $g$ and it is reduced.
\end{enumerate}
\end{definition}

	\begin{fact}[\cite{gut}{[Lemmas 2.2 and 2.3]}]\label{fact_word_1} Let $G = G(\Gamma, \mathfrak{p})$.
	\begin{enumerate}[(1)] 
	\item If the word $a_1^{\alpha_1} \cdots a_k^{\alpha_k}$ spelling the element $g \in G$ is not reduced, then there exist $1 \leq p < q \leq k$ such that $a_p = a_q$ and $a_p$ is adjacent to each vertex $a_{p + 1}, a_{p + 2}, ..., a_{q-1}$.
	\item If $w_1 = a_1^{\alpha_1} \cdots a_k^{\alpha_k}$ and $w_2 = b_1^{\beta_1} \cdots b_k^{\beta_k}$ are normal forms for $g \in G$, then $w_1$ can be transformed into $w_2$ by repetedly swapping the order of adjacent syllables.
\end{enumerate} 
\end{fact}

	
	\begin{definition_proposition}\label{def_prop} Let $\Gamma = (\omega^{\omega}, E)$, with $E$ closed in the Baire space, and $\mathfrak{p}: V \rightarrow \{ p^n : p \text{ prime, } n \geq 1 \} \cup \{ \infty \}$ such that $\mathfrak{p}(\eta)$ depends only on $\eta(0)$. For $0 < n < \omega$, let:
	$$E_n = \{ (\eta, \nu) : \eta, \nu \in \omega^{n} \text{ and there are } (\eta', \nu') \in E \text{ such that } \eta \triangleleft \eta' \text{ and } \nu \triangleleft \nu' \},$$
and $G_n = G((\omega^n, E_n), \mathfrak{p}_n)$, where $\mathfrak{p}_n(\eta) = \mathfrak{p}(\eta')(0)$ for any $\eta \triangleleft \eta'$. For $g \in G(\Gamma, \mathfrak{p}) - \{ e \}$ and $\eta_1^{\alpha_1} \cdots \eta_k^{\alpha_k}$ a word spelling $g$, we define $n(g)$ as the minimal $0 < n < \omega$ such that:
	$$G_n \models (\eta_1 \restriction n)^{\alpha_1} \cdots (\eta_k \restriction n)^{\alpha_k} \neq e.$$
	Finally, for $g \in G(\Gamma, \mathfrak{p}) - \{ e \}$, we define $d(g) = 2^{- n(g)}$, and $d(e) = 0$. 
\end{definition_proposition}

	\begin{proof} We have to show that $n(g)$ does not depend on the choice of the word spelling $g$. So let $\eta_1^{\alpha_1} \cdots \eta_k^{\alpha_k}$ and $\theta_1^{\beta_1} \cdots \theta_m^{\beta_m}$ be words spelling $g \in G$, we want to show that, for every $0 < n < \omega$, the words $(\eta_1 \restriction n)^{\alpha_1} \cdots (\eta_k \restriction n)^{\alpha_k}$ and $(\theta_1 \restriction n)^{\beta_1} \cdots (\theta_m \restriction n)^{\beta_m}$ spell the same element $g' \in G_n$. By Fact \ref{fact_word_1} this is clear, since $\eta_1 E \eta_2$ implies $\eta_1 \restriction n E_n \eta_2 \restriction n$, and $\mathfrak{p}(\eta)$ depends only on $\eta(0)$.
\end{proof}
	
	The following lemma proves Theorem \ref{theorem1}.

	\begin{lemma}\label{lemma_for_th1} Let $\Gamma = (\omega^{\omega}, E)$, with $E$ closed in the Baire space, $\mathfrak{p}: V \rightarrow \{ p^n : p \text{ prime, } n \geq 1 \} \cup \{ \infty \}$ such that $\mathfrak{p}(\eta)$ depends only on $\eta(0)$, and $G = G(\Gamma, \mathfrak{p})$. The function $d: G \times G \rightarrow [0, 1)_{\mathbb{R}}$ such that $d(g, h) = d(g^{-1}h)$, for $d: G \rightarrow [0, 1)_{\mathbb{R}}$ as in Definition/Proposition \ref{def_prop}, is a left-invariant separable group ultrametric extending the usual metric on $\omega^{\omega}$.
\end{lemma}

	\begin{proof} We show that the function $d: G \rightarrow [0, 1)_{\mathbb{R}}$ of Definition/Proposition \ref{def_prop} is an ultranorm, i.e. that it satisfies the following:
	\begin{enumerate}[(i)]
	\item $d(g) = 0$ iff $g = e$;
	\item $d(gh) \leq max\{ d(g), d(h)\}$, for every $g, h \in G$;
	\item $d(g) = d(g^{-1})$, for every $g \in G$.
\end{enumerate}
	We prove (i). Let $g \neq e$ and $\eta_1^{\alpha_1} \cdots \eta_k^{\alpha_k}$ a normal form for $g$. Let $0 < m < \omega$ be such that for every $1 \leq i < j \leq k$ with $\eta_i \neq \eta_j$ we have $\eta_i E \eta_j$ iff $\eta_i \restriction m E_m \eta_j \restriction m$. Then $n(g) \leq m$ and so $2^{- m} \leq 2^{- n(g)} = d(g)$.
	\newline We prove (ii). Without loss of generality $g \neq e$ and $h \neq e$. Let $\eta_1^{\alpha_1} \cdots \eta_k^{\alpha_k}$ and $\theta_1^{\beta_1} \cdots \theta_p^{\beta_p}$ be normal forms for $g$ and $h$, respectively, and let $t = min \{ n(g), n(h) \}$. Then for every $0 < m < t < \omega$ we have:
	$$G_m \models (\eta_1 \restriction m)^{\alpha_1} \cdots (\eta_k \restriction m)^{\alpha_k} (\theta_1 \restriction m)^{\beta_1} \cdots (\theta_p \restriction m)^{\beta_p} = ee = e.$$
	Hence, $t \leq n(gh)$ and so $d(gh) \leq max\{d(g), d(h) \}$.
	\newline We prove (iii). Let $\eta_1^{\alpha_1} \cdots \eta_k^{\alpha_k}$ be a normal form for $g$. It suffices to show that for every $0 < n < \omega$ we have:
	$$G_n \models (\eta_1 \restriction n)^{\alpha_1} \cdots (\eta_k \restriction n)^{\alpha_k} = e \; \Leftrightarrow \; G_n \models (\eta_k \restriction n)^{-\alpha_k} \cdots (\eta_1 \restriction n)^{-\alpha_1} = e,$$
but this is trivially true.
	\newline The fact that $d$ extends the usual metric on $\omega^{\omega}$ is immediate. Thus we are only left to show the separability of $(G, d)$. For every $n < \omega$, define a relation $R_n$ on $G$ by letting $aR_nb$ iff there exist normal forms:
	$$a = \eta_{a, 1}^{\alpha(a, 1)} \cdots \eta_{a, k_a}^{\alpha(a, k_a)} \;\; \text{ and } \;\; b = \eta_{b, 1}^{\beta(b, 1)} \cdots \eta_{b, k_b}^{\beta(b, k_b)}$$
such that $k_a = k_b$, $\alpha(a, \ell) = \beta(b, \ell)$ and $\eta_{a, \ell} \restriction n = \eta_{b, \ell} \restriction n$. Clearly $R_n$ is an equivalence relation on $G$ and it has $\leq \aleph_0$ equivalence classes. For every $n < \omega$, let $X_n$ be a set of representatives of $R_n$ equivalence classes. Then $X = \bigcup_{n < \omega} X_n$ is countable and dense in $(G, d)$, and so it witnesses the separability of $(G, d)$.
\end{proof}

	We need two facts before proving Theorem \ref{embed_for_cyclic}.

\begin{fact}[\cite{gao}{[Theorem 2.1.3]}]\label{completion} Let $G$ be a topological group with compatible left-invariant metric (resp. ultrametric) $d$. Let  $D$ be defined such that:
	$$D(g, h) = d(g, h) + d(g^{-1}, h^{-1}),$$
and $\hat{G}$ the completion of the metric space $(G, D)$. Then the multiplication operation of $G$ extends uniquely onto $\hat{G}$ making $\hat{G}$ into a topological group. Furthermore, there is a unique compatible left-invariant metric (resp. ultrametric) $\hat{d}$ \mbox{on $\hat{G}$ extending $d$.}
\end{fact}

	\begin{definition} We say that a Polish group $G$ is non-Archimedean if it has a neighbourhood base of the identity that consists of open subgroups.
\end{definition}

	\begin{fact}[\cite{kechris}{[Theorem 1.5.1]}]\label{char_non-arch} Let $G$ be Polish. The following are equivalent:
	\begin{enumerate}[(a)]
	\item $G$ is non-Archimedean;
	\item $G$ is isomorphic to a closed subgroup of $Sym(\omega)$;
	\item $G$ admits a compatible left-invariant ultrametric;
	\item $G$ is isomorphic to the automorphism group of a countable first-order structure.
	\end{enumerate}
\end{fact}

	We finally prove Theorem \ref{embed_for_cyclic} and Corollary \ref{random_graph}.
	
		\begin{proof}[Proof of Theorem \ref{embed_for_cyclic}] Suppose that $G(\Gamma, \mathfrak{p})$ is embeddable into a Polish group, then by Proposition \ref{propo_closed} there is a separable metric on $\Gamma$ such that $E_{\Gamma}$ is closed in the induced topology. On the other hand, if there is a separable metric $d$ on $\Gamma$ which induces a topology in which $E_{\Gamma}$ is closed, then using Lemma \ref{metric_implies_ultrametric} we can embed $(\Gamma, \mathfrak{p})$ in a coloured graph on  $\omega^\omega$ which satisfies the assumptions of Lemma \ref{lemma_for_th1}, and so using Facts \ref{completion} and \ref{char_non-arch} we are done.
\end{proof}

	\begin{proof}[Proof of Corollary \ref{random_graph}] As well-known, the automorphism group of the random graph embeds $Sym(\omega)$ (this also follows from the main result of \cite{muller}). Furthermore, in \cite{Sh_Pa_Hall} it is proved that the automorphism group of Hall's universal locally finite group embeds $Sym(\omega)$. Thus, by Theorem \ref{embed_for_cyclic} and Fact \ref{char_non-arch} we are done.
\end{proof}

%
%
%


\begin{thebibliography}{10}


\bibitem{kechris}
Howard Becker and Alexander S. Kechris.
\newblock {\em The Descriptive Set Theory of Polish Group Actions}.
\newblock London Math. Soc. Lecture Notes Ser. 232, Cambridge University Press, 1996.

\bibitem{ding_gao}
Longyung Ding and Su Gao.
\newblock {\em Graev Metric Groups and Polishable Subgroups}.
\newblock Adv. Math. {\bf 213} (2007), no. 2, 887-901. 

\bibitem{ding_gao_1}
Longyung Ding and Su Gao.
\newblock {\em New Metrics on Free Groups}.
\newblock Topology Appl. {\bf 154} (2007), no. 2, 410-420. 

\bibitem{gao_1}
Su Gao.
\newblock {\em Graev Ultrametrics and Surjectively Universal non-Archimedean Polish Groups}.
\newblock Topology Appl. {\bf 160} (2013), no. 6, 862-870. 

\bibitem{gao}
Su Gao.
\newblock {\em Invariant Descriptive Set Theory}.
\newblock Chapman \& Hall/CRC Pure and Applied Mathematics. Taylor \& Francis, 2008.

\bibitem{green}
Elisabeth R. Green.
\newblock {\em Graph Products}.
\newblock PhD thesis, University of Warwick, 1991. 

\bibitem{gut}
Mauricio Gutierrez, Adam Piggott and Kim Ruane.
\newblock {\em On the Automorphism Group of a Graph Product of Abelian Groups}.
\newblock Groups, Geom. Dyn. {\bf 6} (2012), 125-153.

\bibitem{muller}
Isabel M\"uller.
\newblock {\em Fra\"iss\'e Structures with Universal Automorphism Groups}.
\newblock J. Algebra, to appear.

\bibitem{shelah&paolini}
Gianluca Paolini and Saharon Shelah.
\newblock {\em Polish Group Topologies for Graph Product of Cyclic Groups}.
\newblock Submitted. Available on the arXiv.

\bibitem{Sh_Pa_Hall}
Gianluca Paolini and Saharon Shelah.
\newblock {\em The Automorphism Group of Hall's Universal Group}.
\newblock Submitted. Available on the arXiv.

\end{thebibliography}
\end{document}